\documentclass[10pt]{amsart}
\usepackage{amssymb}
\usepackage[all,cmtip]{xy}

\usepackage[T1]{fontenc} 
\usepackage[bookmarks=false]{hyperref}

\renewcommand{\a}{\alpha}

\newcommand{\s}{\sigma}

\renewcommand{\o}{\omega}

\renewcommand{\P}{\mathbb{P}}

\newcommand{\SF}{{\mathcal{F}}}

\newcommand{\SH}{{\mathcal{H}}}

\newcommand{\Z}{\mathbb{Z}}
\newcommand{\C}{\mathbb{C}}

\renewcommand{\H}{\mathbb{H}}

\newcommand{\N}{\mathbb{N}}
\newcommand{\R}{\mathbb{R}}

\renewcommand{\S}{\mathbb{S}}
\newcommand{\T}{\mathbb{T}}

\newcommand{\Cont}{\operatorname{Cont}}
\newcommand{\SCont}{\operatorname{SCont}}
\newcommand{\Ham}{\operatorname{Ham}}

\newcommand{\CSp}{\operatorname{CSp}}
\newcommand{\Diff}{\operatorname{Diff}}

\DeclareMathOperator{\cmfd}{(\mathnormal{M}, \xi=\mathrm{ker}\,\alpha)}
\DeclareMathOperator{\lieg}{\mathfrak{g}}
\DeclareMathOperator{\curv}{\mathnormal{K}^{\sigma}}
\DeclareMathOperator{\chwup}{\overline{\chi}_{\mathnormal{r}}}

\DeclareMathOperator{\xis}{\xi_{\mathrm{st}}}
\DeclareMathOperator{\als}{\alpha_{\mathrm{st}}}
\DeclareMathOperator{\tfix}{\mathbb{T}_{\mathnormal{p}}}
\DeclareMathOperator{\ev}{\mathit{ev}}
\DeclareMathOperator{\stfl}{\mathrm{V}_{2}(\mathbb{R}^{3})}

\newtheorem{proposition}{Proposition}
\newtheorem{theorem}[proposition]{Theorem}

\newtheorem{lemma}[proposition]{Lemma}

\newtheorem{corollary}[proposition]{Corollary}

\theoremstyle{definition}
\newtheorem{definition}[proposition]{Definition}
\newtheorem{example}[proposition]{Example}

\newtheorem{remark}[proposition]{Remark}
\newtheorem{remark*}{Remark}

\begin{document}

\title{Chern--Weil theory and the group of strict contactomorphisms}

\author{Roger Casals}
\address{Instituto de Ciencias Matem\'aticas CSIC--UAM--UCM--UC3M,
C. Nicol\'as Cabrera, 13--15, 28049, Madrid, Spain.}
\email{casals.roger@icmat.es}

\author{Old\v{r}ich Sp\'{a}\v{c}il}
\address{Department of Mathematics, University College London, Gower Street, WC1E 6BT, London, UK \emph{and} Institute of Pure and Applied Mathematics, University of Aberdeen, Scotland, UK.}
\email{o.spacil@ucl.ac.uk}

\begin{abstract}
In this paper we study the groups of contactomorphisms of a closed contact manifold from a topological viewpoint. First we construct examples of contact forms on spheres whose Reeb flow has a dense orbit. Then we show that the unitary group $\mathrm{U}(n+1)$ is homotopically essential in the group of contactomorphisms of the standard contact sphere $\S^{2n+1}$ and present a proof of the homotopy equivalence $\Cont(\S^{3}, \xis) \simeq \mathrm{U}(2)$. In the second part of the paper we focus on the group of strict contactomorphisms -- using the framework of Chern--Weil theory we introduce and study contact characteristic classes analogous to the Reznikov Hamiltonian classes in symplectic topology. We carry out several explicit calculations illustrating the non-triviality of these contact characteristic classes.
\end{abstract}
\maketitle


\section{Introduction}\label{sec:intro}
The study of symplectic and contact topology essentially began with V.\,I.\,Arnol'd in \cite{Ar}. The work of many researchers has established contact topology as a well-founded theory with meaningful results \cite{Be,E2,G1,Gr,Ho}. There have also been many significant applications to other areas of mathematics such as low-dimensional topology \cite{KM,OS}.\\

The problem of existence of contact structures on an almost contact manifold has been solved \cite{BEM,CPP,Et,Lu,Ma} and there have also been partial advances in the classification \cite{E1,G2,H1,H2}. However the two groups of transformations associated to a contact form, the strict contactomorphism group and the contactomorphism group, are not well understood. This article contains relevant results and examples concerning such groups and provides tools for a more systematic study of the group of strict contact transformations.\\

Our main example in the first part of the paper is the standard contact sphere -- we recall its definition in Section \ref{sec:preliminaries} and then study the contact manifold in Sections \ref{sec:hi} and \ref{sec:dyn}. First, using simple algebraic topology tools we prove the following:
\begin{theorem}\label{thm:sph1}
The inclusion $i\colon \mathrm{U}(n+1)\longrightarrow\Cont(\S^{2n+1},\xis)$ induces an injection on the homotopy groups
$\pi_{*}i\colon \pi_{*} \mathrm{U}(n+1)\longrightarrow\pi_{*}\Cont(\S^{2n+1},\xis)$.
\end{theorem}

As a consequence, the group of contactomorphisms of the standard sphere is not homotopically trivial. It is of interest to compare the group of contact transformations with that of strict contact transformations: the choice of a contact form introduces a dynamical flavour to the contact topology framework making the situation more rigid. For example we have the following result together with the corollary below:
\begin{theorem}\label{thm:sph2}
The standard contact form $\als\in\Omega^1(\S^{2n+1})$ can be $C^1$--perturbed to a contact form $\alpha$ such that $\xis=\ker\alpha$ and the strict contact manifold $(\S^{2n+1},\ker\alpha)$ admits a dense Reeb orbit.
\end{theorem}

The techniques employed to conclude Theorem \ref{thm:sph2} belong to the classical theory of dynamical systems and allow further conclusions than the one just stated. They can be, in particular, applied to other contact manifolds, though we only focus on the standard contact spheres $(\S^{2n+1},\xis)$ as an illustrative case. We shall also provide an alternative argument in the case of $\S^3$. Theorems \ref{thm:sph1} and \ref{thm:sph2} imply:

\begin{corollary}\label{cor:strict}
There exist contact forms $\alpha_0$ and $\alpha_1$ defining the standard contact structure on $\S^{2n+1}$ such that the map on the homotopy groups induced by the inclusions
$$j_0\colon \SCont(\S^{2n+1},\alpha_0)\longrightarrow\Cont(\S^{2n+1},\xis)$$
$$j_1\colon \SCont(\S^{2n+1},\alpha_1)\longrightarrow\Cont(\S^{2n+1},\xis)$$
is trivial for $j_{0}$ while the image of $\pi_*j_1$ contains $\pi_{*}\mathrm{U}(n+1)$.
\end{corollary}

Furthermore we include a couple of explicit computations of the homotopy type of groups of contact transformations for low-dimensional spheres. In Subsection \ref{ssec:5SPH} we show that the standard unitary action induces homotopy equivalences $\SCont(\S^3,\als)\simeq \mathrm{U}(2)$ and $\SCont(\S^5,\als)\simeq \mathrm{U}(3)$. 
Then in Section \ref{sec:3SPH} we explain in detail how to use \cite{E2,Ha} to deduce the following folklore statement:\footnote{This result is known to a number of experts in the field, but it seems that a proof has not yet appeared in the literature.}
\begin{theorem}\label{thm:3SPH}
Consider the standard contact 3--sphere $(\S^3,\xis)$. Then
\begin{itemize}
\item[1.] the homotopy type of the space of positive tight contact structures on $\mathbb{S}^{3}$ is that of $\S^2$ and
\item[2.] the action of $\mathrm{U}(2)$ on $(\S^{3}, \xis)$ induces a homotopy equivalence $i\colon \mathrm{U}(2)\longrightarrow\Cont(\mathbb{S}^{3}, \xis)$.
\end{itemize}
\end{theorem}

The two statements are essentially equivalent since for a given contact manifold $(M, \xi)$ the space of contact structures on $M$ and the group $\Cont(M, \xi)$ are related by a locally trivial fibration induced by the Gray stability property (as explained by P.~Massot, the local triviality follows from the Cerf criterion). Little is in fact known about the topology of the space of contact structures, see \cite{CP,GG,Bo}. As a corollary of the above we also obtain the following interesting result.

\begin{corollary}\label{cor:3SPH}
The inclusion $k\colon \SCont(\S^3, \als) \longrightarrow \Cont(\S^3, \xis)$ is a homotopy equivalence.
\end{corollary}

Similar results can be proven for specific contact manifolds, particularly 3--folds. However, the corresponding results require ad hoc arguments and there is a priori no methodical approach to deduce properties of the group of (strict) contact transformations. This naturally leads to the second part of this article.\\

Sections \ref{sec:chernweil} and \ref{sec:CWCont} deal with a systematic procedure to detect non-trivial cohomology classes for the classifying space of the group of strict contactomorphisms. This is achieved by means of Chern--Weil theory, a differential geometric technique to construct characteristic classes for principal $G$--bundles with $G$ a Lie group. Classically this has been done for finite dimensional Lie groups but the framework extends to the case of infinite dimensional Fr\'echet Lie groups or even convenient Lie groups (in the sense of the convenient calculus of \cite{KrMi}). The Chern--Weil method is based on the existence of invariant polynomials on the Lie algebra $\lieg=$ Lie($G$) but these might not necessarily exist for an arbitrary Lie group $G$. In particular it seems hard to find invariant polynomials for various groups of diffeomorphisms -- see \cite{Ro} for a recent survey.\\

An exceptional case of a diffeomorphism group where a series of invariant polynomials is known to exist is the Hamiltonian group $\Ham(B, \omega)$ of a closed symplectic manifold $(B, \omega)$ -- the reader is referred to \cite{Re} for a heuristic treatment of Chern--Weil theory in this case. We will see in Section \ref{sec:chernweil} that an analogous series of invariant polynomials exists on the Lie algebra of the strict contactomorphism group $\SCont(M, \alpha)$ of a closed contact manifold $(M, \alpha)$. Thus the machinery of Chern--Weil theory can be applied to the group of strict contact transformations producing characteristic classes called the (contact) Reznikov classes. Using these we can conclude for instance the following: 

\begin{theorem}\label{thm:reg_chlgy}
Let $\cmfd$ be a closed connected contact manifold such that the Reeb flow of $\a$ generates a free $\S^{1}$--action and consider the homomorphism $\varphi\colon \S^{1} \longrightarrow \SCont(M, \alpha)$ given by this circle action. Then the induced map between the cohomology of the classifying spaces $$(B\varphi)^{*}\colon H^{*}(B\mathrm{SCont}(M, \alpha); \mathbb{R}) \longrightarrow H^{*}(B\mathbb{S}^{1}; \mathbb{R}) \cong \mathbb{R}[\mathnormal{e}]$$ is surjective.
\end{theorem}

The preimage of $\mathnormal{e}$ under the above surjection can be described explicitly, which allows us to deduce a simple corollary:

\begin{corollary}\label{cor:reg_htpy}
Let $\cmfd$ be a closed connected contact manifold such that the Reeb flow of $\a$ generates a free $\S^{1}$--action. Then the homotopy class of this flow in $\pi_{1}(\SCont(M, \alpha))$ is of infinite order.
\end{corollary}

In Section \ref{sec:CWCont} we discuss several other results similar to Theorem \ref{thm:reg_chlgy} such as the following:

\begin{proposition}\label{prop:evenRez}
Let $(M, \xi)$ be a closed connected contact manifold with a non-trivial action of a connected compact Lie group $G$ by contactomorphisms. Then for any $G$--invariant contact form $\alpha$ and $l\in \mathbb{N}$ the Reznikov class $\chi_{2l} \in H^{4l}(B\mathrm{Cont}(M, \alpha); \mathbb{R})$ is non-trivial. 
\end{proposition}

The previous can be improved for a special class of contact manifolds known as $K$--contact manifolds:

\begin{proposition}\label{prop:oddRez} Let $\cmfd$ be a closed connected $K$--contact manifold with a $K$--contact form $\alpha$. Then all the Reznikov classes $\chi_{k} \in H^{2k}(B\mathrm{Cont}(M, \alpha); \mathbb{R}), k\in \N$, are non-trivial.
\end{proposition}

This article is organized as follows. Section \ref{sec:preliminaries} contains the basic definitions and examples needed for the paper. In Section \ref{sec:hi} we give a proof of Theorem \ref{thm:sph1}. Section \ref{sec:dyn} deals with the group of strict contactomorphisms of the standard contact spheres and in particular we explain Theorem \ref{thm:sph2} and Corollary \ref{cor:strict}. Section \ref{sec:3SPH} is devoted to the proof of Theorem \ref{thm:3SPH}, i.e.~the description of the homotopy type of the group of contactomorphisms for the standard contact 3--sphere. In Section \ref{sec:chernweil} we recall the framework of Chern--Weil theory and shortly discuss the infinite dimensional setting. The last Section \ref{sec:CWCont} details the results obtained in the case of the group of strict contact transformations.\\

{\bf Acknowledgements.} We are grateful to R.\,Hepworth, P.\,Massot, M.\,McLean, S.\,M\"{u}ller, F.\,Presas, E.\,Shelukhin and C.\,Wendl for useful comments and suggestions. In particular, after posting a draft of our paper on the arXiv we were informed by S.\,M\"{u}ller that he together with P.\,Spaeth proved a result similar to Theorem \ref{thm:sph2} for some unit cotangent bundles, please see \cite[Proposition 2.1]{MuSp}.\\

This work stems from a discussion in January 2013 during the VII Workshop on Symplectic Geometry, Contact Geometry and Interactions in Les Diablerets, Switzerland, and the first author would like to acknowledge the second author for triggering this initial discussion. \\

The paper contains some results from the second author's PhD thesis \cite{Sp}. On this occasion he would like to thank J.\,K\k{e}dra for his patient supervision and support. He would also like to thank M.\,Guest and Y.\,Maeda for inviting him to the UK--Japan Mathematical Forum, Tokyo, February 2013 and giving him the opportunity to present his results there for the first time. \\

The present work is part of the authors activities within CAST, a Research Network Program of the European Science Foundation. The first author is supported by the Spanish National Research Project MTM2010--17389. The second author's work is funded by a Leverhulme Trust Research Project Grant.


\section{Preliminaries}\label{sec:preliminaries}

In this section we give the basic definitions and discuss several examples in preparation for the rest of the paper.

\subsection{Definitions}\label{ssec:definitions}
A \emph{contact structure} on a smooth $(2n+1)$--dimensional manifold $M$ is a completely non-integrable tangent hyperplane field $\xi$. The complete non-integrability of $\xi$ can be expressed by $\alpha\wedge d\alpha^n\neq 0$ (pointwise everywhere) for a 1--form $\alpha$ locally defining $\xi$, i.e.~$\xi=\ker\alpha$. Such an $\a$ is called a \emph{contact form}.\\

Note that the contact condition can be equivalently described by saying that the $2$--form $d\a$ restricts to a non-degenerate bilinear form on $\xi = \ker\a$. We will always assume that $M$ is oriented and that $\xi$ is positively cooriented, that is there exists a globally defined $\alpha$ such that $\a\wedge (d\a)^{n}$ is a positive volume form on $M$.\\

{\bf The standard contact sphere}. Let us consider the complex space $(\C^{n+1},h)$ with coordinates $(z_0,z_1,\ldots,z_n)$ and its standard Hermitian form $h$. The unit sphere $$\S^{2n+1}=\{z\in\C^{n+1}:\|z\|=1\}$$ is a real hypersurface transverse to the vector field $\partial_r=z_0\partial_{z_0}+\ldots+z_n\partial_{z_n}$. It follows that the 1--form
$$\als=\frac{i}{4}\left(\sum_{i=0}^nz_i d\overline{z}_i-\overline{z}_i dz_i\right)$$
restricted to the hypersurface $\S^{2n+1}$ satisfies $\als\wedge (d\alpha_{\mathrm{st}})^n>0$, hence $(\S^{2n+1},\ker\als)$ is a contact manifold. The contact structure $\xis=\ker\als$ is called the \emph{standard contact structure} on $\S^{2n+1}$ and the defining form $\als$ is the \emph{standard contact form}. Intrinsically, the hyperplane field can be described as $\xis = T\S^{2n+1}\cap i(T\S^{2n+1})$ and $\als$ is its unique (up to scalar) $\mathrm{U}(n+1)$--invariant defining 1--form. This contact structure features in Theorems \ref{thm:sph1},\ref{thm:sph2} and \ref{thm:3SPH}, Proposition \ref{prop:sphsur} and Corollaries \ref{cor:3SPH} and \ref{cor:5SPH}.\\

{\bf The space of contact elements}. Given a smooth manifold $M$ the projectivized cotangent bundle $\P(T^*M)$ is endowed with a canonical contact structure $\xi_{\mathrm{can}}$. An element of $\P(T_{p}^*M)$, called a \emph{contact element} at $p\in M$, is identified with a hyperplane in the tangent space $T_{p}M$. Then the contact structure $\xi_{\mathrm{can}}$ is defined as follows: the velocity vector of a motion of a contact element belongs to $\xi_{\mathrm{can}}$ if and only if the projection of the velocity vector to the point of contact belongs to the contact element itself. This contact structure appears in the proof of Theorem \ref{thm:sph2}, see Section \ref{sec:dyn}.\\ 

{\bf Prequantization spaces}. Let $(B,\omega)$ be a symplectic manifold with an integral symplectic form $\omega\in H^2(B;\Z)$. Since $H^2(B;\Z)\cong [B,B\mathrm{U}(1)]$, there exists a complex line bundle $L_\omega$ endowed with a Hermitian metric with curvature $-i\omega$. Let $\a$ be a principal connection $1$--form for such a Hermitian line bundle. Then the kernel of $\a$ defines a contact structure on the total space of the circle bundle $\S(L_\omega)\longrightarrow B$. The contact manifold $(\S(L_{\o}), \a)$ is called the \emph{prequantization} of $(B,\omega)$. Note that the Reeb flow of $\a$ (defined below) coincides with the circle fibres of the projection and so the flow defines a free $S^{1}$--action on $\S(L_\omega)$ by strict contactomorphisms. In fact any closed contact manifold admitting a free $S^{1}$--action by contactomorphisms can be obtained as a prequantization of some closed symplectic manifold. Prequantization spaces appear in Sections \ref{sec:dyn} and \ref{sec:CWCont}.


\subsection{Contact Hamiltonians}\label{ssec:contactham} There are two groups of symmetries associated to a contact manifold $(M,\xi=\ker\alpha)$: the group of \emph{strict contact transformations} and the group of \emph{contact transformations}. The former depends on the choice of a contact form $\alpha$ whereas the latter does not. Their definitions read as follows:
$$\SCont(M,\alpha)=\{f\in\Diff(M):f^*\alpha=\alpha\},$$
$$\Cont(M,\xi)=\{f\in\Diff(M):f_*\xi=\xi\}.$$

Both groups endowed with the smooth compact-open topology can be given the structure of an infinite dimensional Fr\'echet or convenient Lie group (see \cite[Theorem 43.19]{KrMi} and \cite{Sp}). The tangent space of $\Cont(M, \xi)$ at the identity, hence its Lie algebra, consists of the \emph{contact vector fields} on $M$, i.e.~ vector fields $X$ such that $\mathcal{L}_{X}\alpha = \mu\cdot \alpha$ for some smooth function $\mu\colon M \longrightarrow \mathbb{R}$. In the case of $\SCont(M, \a)$ the tangent space consists of such vector fields with $\mu \equiv 0$, which we call \emph{strictly contact}. The Lie bracket is (minus) the standard Lie bracket of vector fields.\\

There is a different description of the Lie algebras in terms of Hamiltonians. First observe that given a contact form $\alpha$, there exists a unique vector field $R$ on $M$ defined by 
$$i_R\alpha=1,\quad i_Rd\alpha=0.$$
It is called the \emph{Reeb vector field} associated to $\alpha$. The \emph{Reeb flow} of $\alpha$ is the flow of the Reeb field $R$. Let now $X$ be a contact vector field on $M$ an put $H = \alpha(X)$. Then the condition on the Lie derivative of $\a$ is equivalent to the following pair of equations:
$$i_X\alpha=H,\quad i_{X}d\alpha=-dH+i_{R}dH\cdot \alpha.$$
In fact, given a function $H\in C^\infty(M)$ there exists a unique contact vector field $X$ satisfying the above equations. Therefore the map $X \longmapsto \alpha(X) =: H_{X}$ identifies the Lie algebra $\mathrm{Lie}\Cont(M, \xi)$ with the algebra $C^{\infty}(M)$ of smooth functions on $M$. The function $H_{X}$ is called the (contact) \emph{Hamiltonian} of the contact vector field $X$. Note that it depends on the choice of $\a$.\\

Restricting to the Lie algebra of $\SCont(M, \a)$ we can see that a contact vector field $X$ is strictly contact if and only if $i_{R}dH_{X} \equiv 0$, that is if and only if its Hamiltonian $H_{X}$ is a function invariant under the Reeb flow of $\a$. The Lie algebra $\mathrm{Lie}\SCont(M, \a)$ is thus identified with the algebra $C^{\infty}(M)^{R}$ of smooth functions on $M$ invariant under the Reeb flow.\\

Finally, under the above identifications the Lie bracket becomes $$[H_{X}, H_{Y}] = i_{R}(dH_{X})\cdot H_{Y} - i_{X}(dH_{Y})$$ and the adjoint action of $g\in \Cont(M, \xi)$ is given by $\mathrm{Ad}_{g}H = (\lambda\cdot H)\circ g^{-1}$, where $g^{*}\a = \lambda\cdot \a$.\\

The material presented in this section should be sufficient for a good understanding of the main part of the paper.


\section{The inclusion of unitary transformations}\label{sec:hi}

This short section is devoted to the proof of Theorem \ref{thm:sph1}, which is concerned with the standard contact sphere $(\S^{2n+1},\xis)$ introduced in Section \ref{sec:preliminaries}. Since $\xis = T\S^{2n+1}\cap iT\S^{2n+1}$, the group of unitary transformations $\mathrm{U}(n+1)$ acts on $\S^{2n+1}$ preserving $\xis$ and we obtain an inclusion
$$i\colon \mathrm{U}(n+1)\longrightarrow\Cont(\S^{2n+1},\xis).$$
Theorem \ref{thm:sph1} says that the induced map $\pi_*(i)$ on homotopy groups is injective.\\

{\bf Proof of Theorem \ref{thm:sph1}:} Let us consider the space of conformal symplectic frames of $\xis$ at a point $x\in\S^{2n+1}$, that is
$$\mathrm{Fr}_{x}\xis := \{ \tau\colon \mathbb{R}^{2n} \to (\xis)_{x}\,|\, \tau\text{ is a linear isomorphism and } \exists c\in\mathbb{R}_{+}\colon \tau^{*}d\alpha|_{(\xis)_{x}} = c\cdot \omega_{0} \}.$$

Put $\mathrm{Fr}\xis = \bigcup_{x\in \S^{2n+1}} \mathrm{Fr}_{x}\xis$ and denote the natural projection by $p\colon \mathrm{Fr}\xis\longrightarrow\S^{2n+1}$. The conformal symplectic group $$\CSp(2n, \R) := \{g\colon \mathbb{R}^{2n} \to \mathbb{R}^{2n}\,|\, g\text{ is a linear isomorphism and } \exists c\in \mathbb{R}_{+}\colon g^{*}\omega_{0} = c\cdot \omega_{0}\}$$ acts on $\mathrm{Fr}\xis$ from the right by pre-composition of mappings. This action turns $p$ into a principal $\CSp(2n, \R)$--bundle. Note also that the group $\Cont(\S^{2n+1},\xis)$ acts on $\mathrm{Fr}\xis$ from the left via post-composition with the tangent map.\\

Fix a point $x_{0}\in\S^{2n+1}$ with a conformal symplectic frame $\tau_{0}\colon \R^{2n}\longrightarrow (\xis)_{x_{0}}$ at $x_0$ and consider the evaluation map $$\mathrm{ev}\colon \Cont(\S^{2n+1},\xis)\longrightarrow \mathrm{Fr} \xis , \quad \mathrm{ev}(f) = T_{x_{0}}f \circ \tau_{0}.$$
The pullback $i^*\mathrm{ev}$ of the evaluation map via the inclusion $i$ fits into the following diagram
$$\xymatrix{
\CSp(2n,\R) \ar[r] & \mathrm{Fr} \xis \ar[r]^{p} & \S^{2n+1} \\
\mathrm{U}(n) \ar[r] \ar[u] & \mathrm{U}(n+1) \ar[r] \ar[u]^{i^*\mathrm{ev}} & \S^{2n+1} \ar[u]_{\text{id}_{\S^{2n+1}}}
}
$$
Here $\S^{2n+1}$ is viewed as the homogeneous space $\mathrm{U}(n+1)/\mathrm{U}(n)$ and $i^*\mathrm{ev}$ is then a morphism of principal bundles (with different structure groups). Restricted to the fibres $i^*\mathrm{ev}$ yields an inclusion of a maximal compact subgroup $\mathrm{U}(n)$ into $\CSp(2n,\R)$ and thus it is a homotopy equivalence. The base map, being the identity map, is also a homotopy equivalence. Then the five lemma applied to the induced maps between the long exact sequences of homotopy groups for the two fibrations implies that $i^*\mathrm{ev}$ is a weak homotopy equivalence, i.e.~$\pi_{*}\mathrm{ev}\circ \pi_{*}i$ is an isomorphism in all degrees. It follows in particular that the first map $\pi_{*}i\colon \pi_{*}\mathrm{U}(n+1)\longrightarrow\pi_{*}\Cont(\S^{2n+1}, \xis)$ is injective as desired. \qed \\

The unitary transformations preserve even the standard contact form $\als$ and so also the inclusion $\mathrm{U}(n+1)\longrightarrow \SCont(\S^{2n+1},\als)$ induces an injection on homotopy groups. Moreover, Theorem \ref{thm:sph1} implies that the flow of the standard Reeb vector field on the sphere (tangent to the Hopf circles)  produces a non-trivial loop in the fundamental group of $\Cont(\S^{2n+1},\xis)$ -- this was also proven with different methods in \cite{CP}. There exists another argument, also known to E.\,Giroux and V.\,Ginzburg, based on Maslov-type indices to prove the non-triviality of certain homotopy classes. \\

The method used to conclude Theorem \ref{thm:sph1} can be applied to some other contact manifolds. The essential tool is to use the evaluation map and the space of conformal symplectic frames in order to deduce the non-triviality of a given class of contact transformations.


\section{Strict contactomorphisms on spheres}\label{sec:dyn}

The existence of invariant polynomials in the Lie algebra of the group of contact transformations is central for the Chern--Weil methods. This allows us to study the subgroups $\SCont(M,\alpha)$ of strict contact transformations rather than the whole group of contactomorphisms $\Cont(M,\xi)$ for which we have not found any invariant polynomials. The dependence of the strict transformations on $\alpha$ is not simple though and the behaviour of the homotopy type of $\SCont(M,\alpha)$ with respect to a perturbation of $\alpha$ can be almost arbitrary. In this section we provide examples in which the group $\SCont(M,\alpha)$ can be made degenerate just by a $C^\infty$--small perturbation of the $1$--form $\alpha$. The method is fairly general but the family of spheres hereby presented should suffice to illustrate this phenomenon.

\subsection{Ergodic flows}

In this subsection we review the notion of ergodicity of a flow. The relevant context for us is that of a flow in the cotangent bundle, in which case the measure $\mu$ is the Liouville measure of the unit cotangent bundle. The use of 'almost every(where)' in the following paragraphs is understood in the measure theoretic sense.\\

The notion of ergodicity can be considered from the probabilistic viewpoint: a map is ergodic if the process defined by its iteration has time average equal to its total space average almost everywhere. From the geometric perspective this translates into the fact that the points (and subsets) are spread around through the iterations of the map. The exact definition can be stated as follows:

\begin{definition}
A measure preserving flow on a probability space $(X,\SF,\mu)$ is called \emph{ergodic} if its invariant subalgebra is the trivial $\sigma$--algebra generated by the sets of measure $0$.
\end{definition}

The reader is referred to Chapitre 2 in \cite{AA} for details on this subsection. The only relevant property that concerns us is the existence of (many) dense orbits for an ergodic motion.
\begin{lemma}[\cite{AA}]\label{lem:dense}
Almost every orbit of an ergodic flow is dense.
\end{lemma}
This condition is not sufficient for ergodicity and in general not every orbit will be dense. For instance, the geodesic flow on a hyperbolic surface is ergodic (it is even Anosov) and yet there exist infinitely many periodic orbits.\\

The properties of contact Hamiltonians introduced in Section \ref{sec:preliminaries} allow us to prove the following:
\begin{proposition} \label{prop:erg}
Suppose that $(M,\alpha)$, $\mathrm{dim}\,M>1$, is a closed connected contact manifold with a contact form $\a$ which admits a dense Reeb orbit. Then the strict contactomorphism group $\SCont(M,\alpha)$ is homotopy equivalent to a disjoint union of real lines $\R$, one for each connected component.
\end{proposition}
\begin{proof}\footnote{The exposition of this proof has been improved after a suggestion of E.\,Shelukhin and we are grateful for this.}
Let $\gamma$ be a dense orbit of the Reeb flow of $\a$ and let $f\in\SCont(M,\alpha)$ be arbitrary but belonging to the connected component of the identity. Then $f$ is the time-one map of a strictly contact Hamiltonian flow $f_{t}$,~i.e. $f=f_{1}$. The flow is generated by a time dependent Hamiltonian $H\colon [0,1]\times M \to \mathbb{R}$ which for each fixed $t$ is invariant under the Reeb flow of $\a$. It follows that for each $t$, $H(t, -)$ is constant on the dense orbit $\gamma$ and hence on the whole of $M$. Consequently, the flow $f_{t}$ at any given time uniformly follows the Reeb orbits and so $f=f_{1}$ equals the Reeb flow at some time $t_{0}$. In other words, the mapping $\phi\colon t\mapsto \mathrm{Fl}_{t}^{R}$ is a surjective continuous map from $\mathbb{R}$ onto the connected component of the identity in $\SCont(M, \a)$.\\

The map $\phi$ is also injective because the orbit $\gamma$ cannot be periodic (if it was, then $M=S^{1}$). It is not difficult to see that $\phi$ is then a homeomorphism of $\mathbb{R}$ onto the connected component (in fact a diffeomorphism, which can be checked using the smooth charts on $\SCont(M, \a)$, see \cite[Theorem 43.19]{KrMi} or \cite{Sp}). Because for a topological group all connected components are homeomorphic, the claim of the proposition follows.
\end{proof}

In Subsection \ref{ssec:ReebI} we provide an example of a contact form on $\S^3$ with an ergodic Reeb flow, hence with a dense Reeb orbit. Note though that it is not possible to find Anosov flows on $\S^3$. Indeed, the periodic points of an Anosov diffeomorphism grow exponentially whereas the action on the homology of $\S^3$ imposes a finite bound on the number of fixed points. This can also be proven by contact topology methods in the case of $\S^3$, see \cite{EG}.


\subsection{Reeb flows}\label{ssec:ReebI} In this subsection we first explain how to obtain a contact form on $\R\P^3,\S^3$ or $\T^3$ with a dense Reeb orbit and then prove Theorem \ref{thm:sph2}.\\

Let us start with $\S^3$, we want to construct a contact form on $(\S^3,\xis)$ with an ergodic Reeb flow. The $3$--sphere fibres over $\S^2$ as the Hopf fibration and we may think of $\S^3$ as the double cover of the unit cotangent bundle $\S(T^*\S^2)$ of $\S^2$. A geometric way of constructing a vector field on a unit cotangent bundle is to fix a metric and to consider its geodesic flow at the energy level 1. To apply this construction to contact topology we need the following fact:\\

{\bf Fact}: (Huygens) Let $(\Sigma,g)$ be a Riemannian surface. Then the geodesic flow on the unit cotangent bundle $\S(T^*\Sigma)$ is the Reeb flow for the unique $\mathrm{U}(1)$--invariant contact form on $\S(T^*\Sigma)$ considered as the space of cooriented contact elements.\\

This follows from (or rather is) the Huygens principle and the definition of the contact structure on the space of contact elements, see Section \ref{sec:preliminaries} and \cite{AG}.\\

Thus in this case the existence of a dense orbit for a Reeb flow is implied by the existence of a dense orbit for the geodesic flow. On the one hand, it is well known that the geodesic flow of a negative curvature surface is ergodic, see \cite{An} for details. On the other hand, the (non-strict) contactomorphism group is not well understood for the 3--folds $\S(T^*\Sigma)$ and so we cannot compare them. However, as stated in Theorem \ref{thm:3SPH} we at least understand the contactomorphism group of the standard contact $\S^3$ (see Section \ref{sec:3SPH} for details).\\

The possible geodesic flows for surfaces of genus $g=0,1$ have been studied in Riemannian geometry. For instance, the use of focusing caps allowed V.\,Donnay to conclude the following:

\begin{theorem}[\cite{Do}]\label{thm:donnay} Any compact orientable surface can be endowed with a Riemannian metric having an ergodic geodesic flow.
\end{theorem}

Note however that this metric is not a small perturbation of a constant curvature metric. For instance, KAM theory implies that a small perturbation of the flat metric on the torus $\mathbb{T}^2$ contains invariant tori and hence cannot be ergodic.\\

Theorem \ref{thm:donnay} and the Huygens principle yield a contact structure on $\S^3$ with an ergodic Reeb flow. Indeed, consider the two sphere $(\S^2,g)$ endowed with a Riemannian metric whose geodesic flow is ergodic. The unit cotangent bundle $\S(T^*\S^2)$ is diffeomorphic to $\R\P^3=\S^3/\Z_2$ and thus we have an ergodic flow on $\R\P^3$ which coincides with the Reeb flow of a contact form on $\R\P^3$ (due to the Huygens principle). This contact structure is exact symplectically filled by the cotangent bundle $T^*\S^2$ and hence it is tight \cite{Gr}. At this point we obtained a contact form on $\R\P^3$ with (many) dense Reeb orbits. This contact form can be pulled back to the universal cover $\S^3$ yielding a contact form $\alpha$ for the unique (up to isotopy) tight contact structure on $\S^3$. In addition, the Reeb flow of this contact form has many dense orbits.\\

There are various reasons for which we obtain the tight contact structure on $\S^3$. For instance, we may consider an interpolation $\{g_t\}$ between the standard round metric $g_0$ and the metric $g$. The previous construction gives a path of contact forms $\{\alpha_t\}$ such that $\alpha_0$ is the standard contact form on $\S^3$ and $\alpha_1=\alpha$. Gray's stability then implies that $\ker\alpha_0$ is isotopic to $\ker\alpha_1$.\\

In higher dimensions the use of the unit cotangent bundle leads to ergodic flows on manifolds different from the spheres (with the notable exception of $\S^7\longrightarrow\S^4$). A different construction of ergodic flows on lens spaces is presented in \cite{K2,HP}, these methods are however strictly low-dimensional.\\

We will now focus on Theorem \ref{thm:sph2}, i.e.~the existence of a contact form on $\S^{2n+1}$ with an ergodic Reeb flow and thus, by Proposition \ref{prop:erg}, an example where $\SCont(M,\alpha)$ retracts to a set of points. In the argument to prove the theorem we invoke results from \cite{K1}, where the technical construction of the required perturbation is described precisely. \\ 

From the ergodic point of view the essential ingredient available in the standard contact sphere is the existence of two circle symmetries generated by positive Hamiltonians. The hypotheses are also satisfied by other contact manifolds such as toric manifolds. However, the positivity of the associated contact Hamiltonian is a non-trivial contact hypothesis. The precise mathematical framework for Theorem \ref{thm:sph2} is described as follows.\\

{\bf Proof of Theorem \ref{thm:sph2}}: Consider the symplectic manifold $(\R^{2n}\setminus\{0\},\omega_{\mathrm{st}})$. The standard Liouville form $\lambda_{\mathrm{st}}$ induces the contact form $\als$ when restricted to the unit sphere $\S^{2n-1}=\{H(p)=1\}$, with $H(p)=\|p\|^2$ being (twice) the standard kinetic energy. The Reeb vector field associated to $\lambda|_{\{H=1\}}$ is the characteristic foliation direction induced by the symplectic form $\omega_{\mathrm{st}}$ on the hypersurface $\S^{2n-1}$. Since $\omega_{\mathrm{st}}$ is the imaginary part of the standard K\"ahler structure, the leaves coincide with the Hopf circles.\\

The crucial feature we need from the symplectic manifold $(\R^{2n}\setminus\{0\},\omega_{\mathrm{st}})$ is the existence of a $\T^2$--action. Certainly $\R^{2n}\setminus\{0\}$ admits an effective Hamiltonian $\T^{n}$--action given by complex scalar multiplication by $z\in\S^1\subset\C^*$ on the $n$ coordinate complex axes of $\C^n\setminus\{0\}=\R^{2n}\setminus\{0\}$. The moment map of each one-parameter subgroup $\S^1\subset\T^n$ corresponds to the radial coordinate (i.e.~kinetic energy) in the corresponding plane of rotation and so it is a non-negative function. Consider the perturbed Hamiltonian
$$H_{\alpha}=\pi\sum_{i=1}^n\alpha_i|z_i|^2,\quad\alpha=(\alpha_1,\ldots,\alpha_n)\in\R_+^{n}$$
and the 2--torus $\T^2=\langle\alpha,(1,0,\ldots,0)\rangle/\Z^n$. This satisfies the hypothesis of Theorem A in \cite{K1}, which states that given an effective Hamiltonian $\T^2$--action on a Liouville manifold one may perturb a Hamiltonian of the form $H_\alpha$ to another Hamiltonian $\SH$ which is $C^\infty$--close to $H_\alpha$ on a compact set and such that the induced flow on the energy levels of $\SH$ is ergodic with respect to the Liouville measure.\\

The first paragraph in this proof can thus be repeated for $\SH$ and we obtain a contact structure on a level set of $\SH$ whose Reeb vector field has an ergodic flow. The $C^\infty$--closeness implies that such an energy set is diffeomorphic to a sphere and then the Gray stability implies that the contact structure is isotopic to that induced by our original $H$. In conclusion, the contact form induced by the Liouville form on a level set of $\SH$ induces the standard contact structure $(\S^{2n-1},\xis)$ and also has an ergodic Reeb flow.\qed \\

The methods of \cite{K1} also apply to other contact manifolds with at least two positive circle symmetries. In this subsection we provided examples of contact forms for which the strict contactomorphism group is as degenerate as possible. In the next subsection we study the opposite case where the strict contactomorphism group contains a homotopically essential subgroup of symmetries.


\subsection{The group $\SCont(\S^5,\alpha_{\mathrm{st}})$} \label{ssec:5SPH} We shall show here that the homotopy type of (the connected component of the identity of)  the strict contactomorphism group $\SCont(\S^5,\alpha_{\mathrm{st}})$ can be computed from the homotopy type of the group of Hamiltonian diffeomorphisms $\mathrm{Ham}(\mathbb{C}\P^{2}, \omega_{\mathrm{FS}})$. This method can be applied to prequantization spaces in general.\\

Let $(M,\a)$ be a closed connected contact manifold with a contact form $\a$ such that its Reeb flow generates a free $\S^{1}$--action. Recall from Section \ref{sec:preliminaries} that $M$ is then the prequantization of some closed integral symplectic manifold $(B, \o)$. In particular, there is a principal $\S^1$--bundle $p\colon M\longrightarrow B$ such that $p^{*}\o = d\a$. Denote by $\mathrm{SCont}_{0}(M, \a)$ the connected component of the identity in $\SCont(M, \a)$. Then we have the following:

\begin{proposition}[{\cite[Proposition 4]{Vi}}] \label{prop:Vizman} There exists a smooth principal $\S^{1}$--bundle
$$q\colon \mathrm{SCont}_{0}(\mathnormal{M}, \alpha)\longrightarrow\Ham(B, \o)$$
such that the $\S^{1}$--fibre is given by the Reeb flow of $\alpha$. Furthermore, the map $q$ is in fact a homomorphism of Lie groups whose kernel is the circle fibre.
\end{proposition}
\begin{proof}
Any $\alpha$-strict contactomorphism of $M$ commutes with the Reeb flow of $\alpha$, i.e.~it preserves the $\S^{1}$--fibres of the bundle $p\colon M \longrightarrow B$. Therefore every such contactomorphism descends to a diffeomorphism of $B$ that is in fact a Hamiltonian diffeomorphism of $(B, \omega)$. This projection then defines the desired smooth principal $\S^{1}$--bundle.
\end{proof}

The homotopy type of $\Ham(B, \o)$ is known in some cases. For example, the standard action of the projective unitary group  $\mathrm{PU}(2)$ on $(\mathbb{C}\P^1, \omega_{\mathrm{FS}})$, where $\omega_{\mathrm{FS}}$ is the Fubini--Study symplectic form, induces a homotopy equivalence $\psi\colon \mathrm{PU}(2) \longrightarrow \mathrm{Ham}(\mathbb{C}\P^1, \omega_{\mathrm{FS}})$. The group of Hamiltonian diffeomorphisms of $\C\P^2$ is computed in \cite{Gr} (or see \cite[Theorem 9.5.3]{MS}):

\begin{theorem}[\cite{Gr}]\label{thm:gr} The standard action of $\mathrm{PU}(3)$ on $\mathbb{C}\P^2$ induces a homotopy equivalence
$$\psi\colon \mathrm{PU}(3) \longrightarrow \mathrm{Ham}(\mathbb{C}\P^2, \omega_{\mathrm{FS}}).$$
\end{theorem}

These two results allow us to compute the homotopy type of $\SCont_{0}(\S^5,\als)$.

\begin{corollary} \label{cor:5SPH}
The standard action of the unitary group $\mathrm{U}(3)$ on $(\S^5, \alpha_{\mathrm{st}})$ induces a homotopy equivalence
$\varphi\colon \mathrm{U}(3)\longrightarrow \mathrm{SCont}_{0}(\S^5,\alpha_{\mathrm{st}}).$
\end{corollary}
\begin{proof} The sphere $(\S^5,\a_{\mathrm{st}})$ is the total space of the Hopf fibration $p\colon \S^{5} \longrightarrow \mathbb{C}\P^{2}$ and this fibration can be regarded as the prequantization of $\C\P^2$ with its Fubini--Study symplectic form. Consider the following morphism of circle bundles
\begin{center}
\begin{minipage}{0.5\linewidth}
\xymatrix{ \mathrm{U}(3) \ar[r]^<<<<<{\varphi} \ar[d] & \mathrm{SCont}_{0}(\S^{5}, \als) \ar[d]^{q} \\ \mathrm{PU}(3)  \ar[r]^<<<<<{\psi} & \mathrm{Ham}(\mathbb{C}\P^2, \omega_{\mathrm{FS}}) }
\end{minipage}
\end{center}
where the projection on the right is the principal $\S^{1}$--bundle from Proposition \ref{prop:Vizman} and the one on the left is the quotient map. By Theorem \ref{thm:gr} the base map $\psi$ is a homotopy equivalence. The map $\varphi$ restricted to the circle fibres equals the Reeb flow and hence it is a diffeomorphism (and in particular a homotopy equivalence). The five lemma applied to the corresponding morphisms between the long exact sequences of homotopy groups for the two fibrations implies that $\varphi$ is a weak homotopy equivalence. Since $\mathrm{SCont}_{0}(\S^{5}, \als)$ has the homotopy type of a CW--complex (see Remark \ref{remark:CWtype} below), $\varphi$ is a genuine homotopy equivalence by the Whitehead theorem.
\end{proof}

Note that the same argument applied to the prequantization $(\S^3, \als) \longrightarrow (\C\P^1, \o_{\mathrm{FS}})$ proves that also the inclusion $\mathrm{U}(2)\longrightarrow \SCont_0(\S^3,\als)$ is a homotopy equivalence. If we understood the inclusion $\mathrm{PU}(n+1) \longrightarrow \mathrm{Ham}(\mathbb{C}\P^n, \omega_{\mathrm{FS}})$ for general $n$, we could deduce similar results even for higher dimensional spheres.

\begin{remark}\label{remark:CWtype} For a closed contact manifold $\cmfd$ the groups $\SCont(M, \a)$ and $\Cont(M, \xi)$ endowed with the smooth compact--open topology can be given the structure of a Fr\'{e}chet manifold (see \cite[Theorem 43.19]{KrMi} or \cite{Sp}). In particular, the topology is Hausdorff, metrizable and locally ANR. It follows from \cite[Theorem 5]{Pa} that the groups are ANR. Moreover, the topology is also separable and so this combined with \cite[Theorem 1]{M2} implies that the contactomorphism groups have the homotopy type of a countable CW--complex. 
\end{remark}


\section{The contactomorphism group of $(\S^3,\xi_{\mathrm{st}})$} \label{sec:3SPH}

In this section we shall describe the homotopy type of the group $\Cont(\S^3, \xis)$ and that of the space of tight contact structures on $\S^3$. Orient $\S^{3}$ as the boundary of the unit ball in $\R^4$. There are two ingredients that allow us to compute the homotopy types. The first one is the knowledge of the homotopy type of the group $\Diff_{+}(\S^3)$ of orientation preserving diffeomorphisms of $\S^{3}$.

\begin{theorem}[\cite{Ha}] \label{thm:Hatcher} The standard action of the special orthogonal group $\mathrm{SO}(4)$ on $\S^{3}$ induces a homotopy equivalence $\mathrm{SO}(4)\longrightarrow\Diff_{+}(\S^3)$.
\end{theorem}

Let $\Xi(\S^3)$ be the space of all cooriented contact structures on $\S^3$ and $\mathbb{T} \subseteq \Xi(\S^{3})$ denote the subspace of tight contact structures which induce the positive orientation on $\S^{3}$, i.e.~for any contact form $\alpha$ defining the contact distribution the volume form $\alpha\wedge d\alpha$ is positive. Fix the point $p = (1,0,0,0)\in \S^{3}$ and consider the subspace $\tfix := \{ \xi \in \mathbb{T} \,|\, \xi(p) = \xis(p) \}$ of positive tight contact structures which agree with the standard one at $p$.\\

By \cite[Theorem 2.1.1]{E2} the contact structure $\xis$ is the unique positive tight contact structure on $\S^{3}$ up to isotopy, which according to Gray's stability theorem means that $\mathbb{T}$ is exactly the connected component of $\xis$ in $\Xi(\S^{3})$. The topology of the subspace $\tfix$ turns out to be even simpler.

\begin{theorem}[{\cite[Theorem 2.4.2]{E2}}] \label{thm:Eli} The space $\tfix$ is contractible.
\end{theorem}

This is the second ingredient that allows us to prove Theorem \ref{thm:3SPH}.\\

{\bf Proof of Theorem \ref{thm:3SPH}:} 1. The strategy is to find a fibre bundle $\mathbb{T}\longrightarrow\S^2$ whose typical fibre is the space $\tfix$. Then Theorem \ref{thm:Eli} implies the first claim of the statement.\\

Let us construct such a map. Trivialize the tangent bundle $T\S^3 \subseteq \R^4 \cong \H^{1}$ using the vector fields
\begin{align*}
X(x_{1}, y_{1}, x_{2}, y_{2}) &= (-y_{1}, x_{1}, -y_{2}, x_{2})^{t} = i\cdot(x_1,y_1,x_2,y_2)^{t}, \\
Y(x_{1}, y_{1}, x_{2}, y_{2}) &= (-x_{2}, y_{2}, x_{1}, -y_{1})^{t} = j\cdot(x_1,y_1,x_2,y_2)^{t},\\
Z(x_{1}, y_{1}, x_{2}, y_{2}) &= (-y_{2}, -x_{2}, y_{1}, x_{1})^{t} = k\cdot(x_1,y_1,x_2,y_2)^{t}.
\end{align*}
Note that $X, Y, Z$ are orthonormal with respect to the standard Riemannian metric and so the tangent sphere bundle can be described as
$$\S(T\S^{3}) = \{ aX+bY+cZ \,|\, a,b,c \in \mathbb{R}\colon a^2+b^2+c^2 = 1\} \cong \S^{3}\times \S^{2}.$$

A cooriented (hyper)plane distribution $\xi \subseteq T\S^{3}$ is uniquely determined by its positive unit normal vector field $\mathbf{n}_{\xi}$ that, using the trivialization of $T\S^{3}$ above, defines a unique Gauss map $\mathbf{n}_{\xi}\colon \S^{3}\longrightarrow\S^{2}$. This establishes a bijective correspondence between cooriented plane distributions in $T\S^{3}$ and smooth maps $\S^{3}\longrightarrow\S^{2}$, in which the standard contact structure $\xis$ corresponds to the map $\mathbf{n}_{\xis}\colon p \longmapsto X(p)$.\\

Fix the point $p=(1,0,0,0)$ and consider the evaluation map $\ev\colon \mathbb{T}\longrightarrow\S^{2}$ defined by $\ev(\xi) = \mathbf{n}_{\xi}(p)$. We shall prove that $\ev$ is a fibre bundle with typical fibre $\tfix$ by constructing local trivializations.\\

The local charts are provided by the appropriate manifold of frames. Indeed, consider the Stiefel manifold $\stfl$ of orthonormal 2--frames in $\mathbb{R}^{3}$ and the fibre bundle $\pi\colon \stfl\longrightarrow\S^{2}$ defined by mapping a two--frame $(\mathbf{u},\mathbf{v})$ to its vector product $\mathbf{u}\times \mathbf{v}$, i.e.~the unique unit vector positively orthogonal to both $\mathbf{u}$ and $\mathbf{v}$. Identify $\stfl$ with $\mathrm{SO}(3)$ by assigning to $(\mathbf{u},\mathbf{v})$ the matrix with ordered columns $\mathbf{u}\times \mathbf{v}, \mathbf{u}$ and $\mathbf{v}$. Then the projection $\pi$ sends a matrix $A \in \mathrm{SO}(3)$ to its first column $c_{1}(A) \in \S^{2}$. \\ 

Given a local section $s\colon U \subseteq \S^{2}\longrightarrow\stfl$ of $\pi$, define for all $b\in U$
$$A_{b} := \mathrm{diag}(1, s(b)) \in \mathrm{SO}(4) \subseteq \Diff_{+}(\S^3),$$
where $s(b)\in SO(3)\cong\stfl$. This induces a local trivialization of $\ev\colon \mathbb{T}\longrightarrow \S^{2}$ over $U$ defined as
$$\psi_{s}\colon U \times \tfix\longrightarrow\mathbb{T}, \quad \psi_{s}(b, \xi) = (A_{b})_{*}\xi, \quad \text{where} \quad (A_{b})_{*}\xi(q) = TA_{b}\bigl(\xi(A_{b}^{-1}(q))\bigr), q\in \S^3.$$
Note that $(A_{b})_{*}\xi$ is indeed a positive tight contact structure on $\S^3$.\\

Let us check that $\psi$ is a bundle chart. First the composition $\ev \circ \psi_{s}$ equals the projection onto the first factor:
$$(\ev\circ\psi_s)(b,\xi)=\ev((A_{b})_{*}\xi) = \mathbf{n}_{(A_{b})_{*}\xi}(p) = A_{b}\cdot \mathbf{n}_{\xi}(p) = c_{2}(A_{b}) = c_{1}(s(b)) = \pi(s(b)) = b$$
since $\mathbf{n}_{\xi}(p) = \mathbf{n}_{\xis}(p) = (0,1,0,0)^{t}$. Furthermore, the inverse of $\psi_{s}$ is given by
$$\varphi_{s}\colon \ev^{-1}(U)\longrightarrow U \times \tfix, \quad \varphi_{s}(\eta) = \left(\ev(\eta), \left(A_{\ev(\eta)}^{-1}\right)_{*}\eta \right).$$
Indeed, put $b=\ev(\eta)=\mathbf{n}_{\eta}(p)$. Then we need to verify that $(A_{b}^{-1})_{*}\eta \in \tfix$. Because the matrix $A_{b} \in \mathrm{SO}(4)$ is orthogonal, we have $A_{b}^{-1} = A_{b}^{t}$ and so
$$\mathbf{n}_{(A_{b}^{-1})_{*}\eta}(p) = A_{b}^{-1}\cdot \mathbf{n}_{\eta}(p) = A_{b}^{t}\cdot c_{2}(A_{b}) = (0,1,0,0)^{t} = \mathbf{n}_{\xis}(p).$$
To sum up, we found a system of bundle charts so that the map $\ev\colon \mathbb{T}\longrightarrow\S^{2}$ is a fibre bundle with typical fibre $\tfix$.\\

Since the space $\tfix$ is contractible by Theorem \ref{thm:Eli}, we deduce from the long exact sequence of homotopy groups for a fibration that $\ev$ is a weak homotopy equivalence. The topological space $\mathbb{T}$ has the homotopy type of a CW--complex\footnote{By an argument analogous to that of Remark \ref{remark:CWtype}.} and thus $\ev$ is a genuine homotopy equivalence.\\

2. Consider the map $\Psi\colon \Diff_{+}(\S^3)\longrightarrow \mathbb{T}$ defined by $\Psi(f) = f_{*}\xis$. This is a surjective map by Gray's stability theorem and the fibre over $\xis$ is $\Psi^{-1}(\xis) = \Cont(\S^3, \xis)$. The parametric version of Gray's stability theorem implies that $\Psi$ has the homotopy lifting property with respect to discs and smooth isotopies. By standard arguments for Serre fibrations this implies in particular that
$$\pi_{k}(\Diff_{+}(\S^3), \Cont(\S^3, \xis)) \cong \pi_{k}(\mathbb{T}),\, k\geq 0.$$
In the same spirit of the argument in the first part of the proof we can show that
$$(\ev \circ \Psi)|_{\mathrm{SO}(4)}\colon \mathrm{SO}(4)\longrightarrow \S^{2}$$
is a fibre bundle with fibre $\mathrm{U}(2)$. By Theorem \ref{thm:Hatcher} $\Diff_{+}(\S^3) \simeq \mathrm{SO}(4)$ and so for all $k\geq 0$
\begin{align*} \pi_{k}(\Diff_{+}(\S^3), \mathrm{U}(2)) \cong \pi_{k}(\mathrm{SO}(4), \mathrm{U}(2)) \cong \pi_{k}(\S^{2}) \cong \pi_{k}(\mathbb{T}) \cong \pi_{k}(\Diff_{+}(\S^3), \Cont(\S^3, \xis)).
\end{align*}
Now consider the long exact sequence of homotopy groups associated to the triple $(D, C, \mathrm{U}) := (\Diff_{+}(\S^3), \Cont(\S^3, \xis), \mathrm{U}(2))$:
\begin{align*} \cdots \to \pi_{k+1}(D, \mathrm{U}) \to \pi_{k+1}(D, C)  \to \pi_{k}(C, \mathrm{U}) \to \pi_{k}(D, \mathrm{U}) \to \pi_{k}(D, C) \to \cdots. \end{align*}
This long exact sequence and the above isomorphisms imply that $\pi_{k}(\Cont(\S^3, \xis), \mathrm{U}(2)) = 0$ for all $k$, i.e.~the inclusion $i\colon \mathrm{U}(2) \longrightarrow \Cont(\S^3, \xis)$ is a weak homotopy equivalence. Because $\Cont(\S^3, \xis)$ has the homotopy type of a CW--complex, $i$ is a genuine homotopy equivalence. \qed \\

{\bf Proof of Corollary \ref{cor:3SPH}:} The homomorphism $i\colon\mathrm{U}(2) \longrightarrow \Cont(\S^3, \xis)$ factors through the inclusion of the subgroup $\SCont_{0}(\S^{3}, \alpha_{\text{st}})$ as on the following diagram.
\begin{center}
\begin{minipage}{0.5\linewidth}
\xymatrix{ & \Cont(\S^{3}, \xis ) \\ \mathrm{U}(2)  \ar[r]_<<<<{j} \ar[ur]^{i} & \SCont_{0}(\S^{3}, \alpha_{\text{st}}) \ar[u]_{k}
}
\end{minipage}
\end{center}

The map $i$ is a homotopy equivalence by Theorem \ref{thm:3SPH}, while $j$ is a homotopy equivalence by the argument of Corollary \ref{cor:5SPH}. It follows that $k$ induces an isomorphism on the homotopy groups, i.e.~it is a weak homotopy equivalence. Since the spaces involved have the homotopy type of CW--complexes, $k$ is a genuine homotopy equivalence.\qed\\

The proofs of Theorem \ref{thm:3SPH} and Corollary \ref{cor:3SPH} close the first part of the paper. In the second part we aim to use the techniques of Chern--Weil theory to study the group of strict contact transformations.


\section{Chern--Weil theory}\label{sec:chernweil}

In this section we review the basics of Chern--Weil theory keeping in mind the case of infinite dimensional Lie groups such as $\SCont(M, \a)$. First, in Subsection \ref{ssec:BG} the notions of a classifying space and characteristic classes are briefly recalled. Subsection \ref{ssec:chernweil} then details the constructions of Chern--Weil theory and Subsection \ref{ssec:infgeom} comments on some nuances of the infinite dimensional case.

\subsection{Classifying spaces} \label{ssec:BG} Let $G$ be a topological group. The first important notion we need is the content of the following definition:

\begin{definition}\label{def:uni}
A \emph{universal principal $G$--bundle} is a numerable\footnote{Numerability is a standard condition on fibre bundles, but unfortunately it is not mentioned frequently. A covering $(U_{i})_{i \in I}$ of a topological space $B$ is called \emph{numerable} if there exists a locally finite partition of unity $(\lambda_{i})_{i \in I}$ such that $\mathrm{supp}(\lambda_{i}) \subseteq U_{i}$ for all $i\in I$. A fibre bundle $P\longrightarrow B$ is then called \emph{numerable} if it is locally trivial over a numerable covering of the base $B$. For example, a locally trivial fibre bundle over a paracompact space is numerable.} principal $G$--bundle $EG\longrightarrow BG$ such that for any other numerable principal $G$--bundle $P\longrightarrow B$ there exists, up to a $G$--bundle homotopy, a unique $G$-bundle morphism $\Phi\colon P \longrightarrow EG$. The base space $BG$ is called a \emph{classifying space} of $G$.
\end{definition}

We shall always suppose that principal bundles are numerable. A principal $G$--bundle $P\longrightarrow B$ is then universal if and only if the total space $P$ is contractible as a topological space. The existence of a universal bundle for any topological group $G$ was proved by J.\,Milnor \cite{M1}. He also showed that if the base space $BG$ has the homotopy type of a countable CW--complex, then this homotopy type is unique. In particular, the cohomology ring $H^{*}(BG)$ is independent of the choice of a model for $BG$.\\

Let us furthermore define the notion of a characteristic class of principal $G$--bundles.

\begin{definition}
Let $R$ be a commutative ring. An $R$--valued \emph{characteristic class} on the category $\mathfrak{C}$ of (numerable) principal $G$--bundles is a map
$$\chi \colon \mathfrak{C} \ni (P\longrightarrow B) \longmapsto \chi(P) \in H^{*}(B; R)$$
which is natural with respect to $G$--bundle morphisms, i.e.~for a $G$--bundle morphism 
\begin{center}
\begin{minipage}{0.5\linewidth}
\xymatrix{ P_{1} \ar[r]^{\Phi} \ar[d] & P_{2} \ar[d] \\ B_{1} \ar[r]^{\phi} & B_{2}}
\end{minipage}
\end{center}
and the induced map $\phi^{*}\colon H^{*}(B_{2}; R) \to H^{*}(B_{1}; R)$ we have $\chi(P_{1}) = \phi^{*} \chi(P_{2})$.
\end{definition}

According to Definition \ref{def:uni} for a given a principal $G$--bundle $P\longrightarrow B$ there exists a $G$--bundle morphism
\begin{center}
\begin{minipage}{0.5\linewidth}
\xymatrix{ P \ar[r]^{\Phi} \ar[d] & EG \ar[d] \\ B \ar[r]^{\phi} & BG}
\end{minipage}
\end{center}
The base map $\phi$ of this morphism is called the \emph{classifying map} of $P$ and it is uniquely defined up to homotopy. By the naturality property of characteristic classes we then have $\chi(P)= \phi^{*}\chi(EG)$ and so it follows that $\chi$ is uniquely determined by its value on the universal bundle $EG\longrightarrow BG$. On the other hand, we can fix $\chi_{0} \in H^{*}(BG; R)$ and prescribe $\chi(EG)=\chi_0$, thus defining a characteristic class by $\chi(P) := \phi^{*}\chi_{0}$. Therefore the set, hereafter the ring, of characteristic classes of principal $G$--bundles is isomorphic to the cohomology ring $H^{*}(BG; R)$ of a classifying space $BG$ of the group $G$.\\

Chern--Weil theory is a classical differential geometry framework for construction and realization of such characteristic classes.


\subsection{Chern--Weil theory} \label{ssec:chernweil}
Let $G$ be a Lie group with Lie algebra $\mathfrak{g}$ and a let $P\longrightarrow B$ be a smooth principal $G$--bundle. Denote by $r_{a}$ the right action of $a\in G$ on $P$ and consider a principal connection form $\sigma\colon TP\longrightarrow \mathfrak{g}$ with curvature form $K^{\sigma}\colon TP \times_{P} TP\longrightarrow\mathfrak{g}$. The curvature form is a $\lieg$--valued two--form on $P$ which is $G$--equivariant, i.e.~$(r_{a})^{*}K^{\s} = \mathrm{Ad}_{a^{-1}}\circ K^{\s}$ for all $a\in G$.\\

Moreover, fix a symmetric multilinear form $r\colon \lieg \times \ldots \times \lieg\longrightarrow \mathbb{R}$ on $\lieg$ which is $\mathrm{Ad}_{G}$--invariant, i.e.
$$ r(\mathrm{Ad}_{a}A_{1}, \ldots, \mathrm{Ad}_{a}A_{k}) = r(A_{1}, \ldots, A_{k}) \quad \text{for all} \quad a\in G, A_{1}, \ldots, A_{k} \in \lieg.$$

Such an $r$ is called an \emph{invariant polynomial} on $\lieg$. Now consider the composition

$$\chwup\colon (TP \times_{P} TP) \times_{P} \ldots \times_{P} (TP\times_{P} TP) \xrightarrow{\curv \times \ldots \times \curv} \lieg\times\ldots\times\lieg \stackrel{r}{\longrightarrow} \mathbb{R}.$$

This is a $2k$--form on $P$ which, due to the $G$--equivariance of $K^{\sigma}$, the invariance of $r$ and the Bianchi identity for $K^{\sigma}$, descends to a closed $2k$--form $\underline{\chi}_r$ on the base $B$. Its de\,Rham cohomology class $$\chi_{r}(P) := [\underline{\chi}_r] \in H^{2k}_{\mathrm{deR}}(B)$$ is independent of the choice of the connection form $\sigma$. It follows that the assignment
$$(P\longrightarrow B) \longmapsto \chi_{r}(P)$$
defines a characteristic class for smooth principal $G$--bundles.\\

This construction of characteristic classes is referred to as \emph{Chern--Weil theory}, see e.g.~ \cite{Du} for a detailed description. J.\,L.\,Dupont also shows that the construction can be extended to the universal principal $G$--bundle $EG \longrightarrow BG$ yielding a universal characteristic class $\chi_{r}\in H^{2k}(BG; \mathbb{R})$. In fact if $G$ is a compact Lie group, we obtain all characteristic classes in this way. Denote by $I^{*}(G)$ the set of all invariant polynomials on $\lieg$, which forms a ring under the addition and multiplication (composed with symmetrization) of polynomials. Then we have the following important result.

\begin{theorem}[H.\,Cartan] \label{thm:cartan}
For a compact Lie group $G$ the map $\chi\colon I^{*}(G)\longrightarrow H^{*}(BG; \mathbb{R})$ defined by $r\longmapsto \chi_{r}(EG)$ is an isomorphism of rings.   
\end{theorem}

\begin{example} \label{ex:torus}
Let $G = \mathbb{T}^{n}$ be the $n$-dimensional torus. Because the adjoint action on its Lie algebra $\mathfrak{t} \cong \mathbb{R}^{n}$ is trivial, any polynomial $r$ in $n$ variables determines a characteristic class $\chi_{r}$ of principal $\T^{n}$--bundles. By Cartan's theorem $H^{*}(B\mathbb{T}^{n}; \mathbb{R}) \cong \mathbb{R}[t_{1}, \ldots, t_{n}].$\\

In the case $G=S^{1}$ the generator $t_{1}$ is just the identity map $\text{id}_{\mathbb{R}}$. Then for a principal $S^{1}$--bundle with a principal connection form $\sigma\colon TP\longrightarrow \mathbb{R} \cong \mathrm{Lie}\,S^{1}$ the induced Chern--Weil two--form $\overline{\chi}_{t_{1}}$ equals the curvature form $K^{\sigma}$. In particular, $\chi_{t_{1}}(P) \in H^{2}_{\text{deR}}(B)$ is represented by the projection of $K^{\sigma}$ to the base. The characteristic class $\chi_{t_{1}}(P)$ is called the (real) \emph{Euler class} of $P$.
\end{example}


\subsection{Infinite dimensional geometry} \label{ssec:infgeom}

The machinery of Chern--Weil theory can be extended to a preferred setting of infinite dimensional differential geometry -- either to the setting of Fr\'{e}chet manifolds or to that of convenient calculus of Fr\"{o}licher, Kriegl and Michor \cite{KrMi} (for a rapid introduction to the basic notions of convenient calculus we also suggest \cite{Le}). For technical reasons we prefer the framework of convenient calculus, for the present paper the choice is not essential though. Various groups of diffeomorphisms of smooth finite dimensional manifolds can be endowed with the structure of a convenient Lie group, e.g.~the groups of contact transformations $\Cont(M, \xi)$ (\cite[Theorem 43.19]{KrMi}) and $\SCont(M, \a)$ (\cite{By} or \cite{Sp}) of a closed contact manifold $(M, \xi)$. It thus makes sense to study their Chern--Weil theory.\\

Some major differences between the convenient calculus and finite dimensional geometry arise from topological complexity of the modeling topological vector spaces. For example, there are many possible definitions of differential forms, but fortunately only one seems to be 'correct' in the sense that all the classical operations such as pullback, exterior differentiation and Lie derivative preserve this class of differential forms. Furthermore, the de\,Rham theorem for convenient manifolds is recovered provided that the underlying topology of the smooth manifold is paracompact.\\

The Chern--Weil construction of characteristic classes then works as described above under the additional assumption that the invariant polynomial $r$ is a bounded map with respect to the natural topology on the Lie algebra $\lieg$. J.\,L.\,Dupont's method of extending Chern--Weil theory to the universal bundles is also easily adapted.\\

In the next section we focus on the Chern--Weil theory of the group of strict contact transformations of a closed contact manifold.


\section{Applications to $\SCont(M, \alpha)$}\label{sec:CWCont}

It seems in general difficult to find invariant polynomials on the Lie algebra of a convenient Lie group, in particular the group of diffeomorphisms of a smooth manifold or its subgroups. However, for a closed symplectic manifold a series of invariant polynomials on the Lie algebra of its group of Hamiltonian diffeomorphisms has been discovered and studied by A.\,G.\,Reznikov \cite{Re}. In this section we define analogous invariant polynomials for the groups of strict contactomorphisms.

\subsection{Reznikov classes for $\SCont(M, \a)$} \label{ssec:reznikov}

Let $\cmfd$ be a closed connected contact manifold. Recall from Section \ref{sec:preliminaries} that the Lie algebra $\mathcal{L} := \mathrm{Lie}\SCont(M, \alpha)$ of the group $\SCont(M, \a)$ of strict contactomorphisms of $M$ is the algebra of strictly contact vector fields on $M$. Using the contact form $\a$ it is identified with the algebra $C^{\infty}(M)^{R}$ of smooth functions on $M$ which are invariant under the Reeb flow of $\alpha$. The adjoint action of $f\in \SCont(M, \a)$ on $H\in \mathcal{L}$ is then given by $\mathrm{Ad}_{f}(H) = H\circ f^{-1}$. \\

For $k\in\mathbb{N}$ consider the map $I_{k}\colon \mathcal{L} \times \ldots \times \mathcal{L}\longrightarrow\mathbb{R}$ defined by
\begin{equation}\label{eq:pols}
I_{k}(H_1, \ldots, H_{k}) = \int_{M} H_{1}\cdot \ldots \cdot H_{k}\,\alpha\wedge(d\alpha)^{n}, \quad \text{where} \quad H_{1}, \ldots, H_{k} \in C^{\infty}(M)^{R}. \end{equation}

Because any $f\in \SCont(M,\a)$ preserves $\alpha$, it also preserves the volume form $\alpha\wedge(d\alpha)^{n}$ and so the map $I_{k}$ is $\mathrm{Ad}$--invariant. Moreover, since integration over a closed manifold is a bounded operator,\footnote{This holds in the convenient topology on $\mathcal{L}$, which differs from the smooth compact-open topology. However, the two topologies share the same bounded sets \cite{KrMi,Sp}.} $I_{k}$ is an invariant polynomial on $\mathcal{L}$.\\

Observe that $I_{k}$ is not invariant under the adjoint action of the larger group $\Cont(M, \xi)$ of contactomorphisms of $(M, \xi)$. This is the reason why we restrict to the group of strict contactomorphisms. We define the contact Reznikov classes as follows.

\begin{definition}
Let $\cmfd$ be a closed connected contact manifold and $k\in \mathbb{N}$. The $k$--th \emph{Reznikov class} $\chi_{k} \in H^{2k}(B\mathrm{Cont}(M, \alpha); \mathbb{R})$ is the characteristic class of principal $\SCont(M, \a)$--bundles determined by the invariant polynomial $I_{k}$.
\end{definition}

Note however that we do not know yet whether the Reznikov classes are non-trivial. We will see below that in various cases of contact manifolds admitting an action of a compact Lie group at least some of the Reznikov classes are indeed non-trivial. Let us briefly describe the strategy to show this.\\

Let $\varphi\colon G_{1}\longrightarrow G_{2}$ be a homomorphism of Lie groups and let $\widehat{\varphi}\colon \mathfrak{g}_{1} \to \mathfrak{g}_{2}$ denote the induced homomorphism of Lie algebras. Suppose that we have a bundle morphism of smooth principal $G_1$ and $G_2$--bundles $\Phi\colon P_{1}\longrightarrow P_{2}$ represented by $\varphi$, i.e.~such that $\Phi(p\cdot a) = \Phi(p)\cdot \varphi(a)$ for all $p\in P_{1}$ and $a\in G_{1}$. Suppose also that the bundle $P_{j}$ is equipped with a principal $G_{j}$--connection $\sigma_{j}\in\Omega^1(P_j,\mathfrak{g}_{j})$, for $j=1,2$, such that $\Phi^{*}\sigma_{2} = \widehat{\varphi} \circ \sigma_{1}$. It is straightforward to check that the curvature forms then satisfy the analogous relation $\Phi^{*}K^{\sigma_{2}} = \widehat{\varphi} \circ K^{\sigma_{1}}$.\\


Given an invariant polynomial $r\colon \mathfrak{g}_{2}\times \ldots \times \mathfrak{g}_{2}\longrightarrow \mathbb{R}$ on $\mathfrak{g}_{2}$, we consider its pullback $$\widehat{\varphi}^{*}r \colon \mathfrak{g}_{1} \times \ldots \times \mathfrak{g}_{1} \xrightarrow{\widehat{\varphi}\times \ldots \times \widehat{\varphi}} \mathfrak{g}_{2} \times \ldots \times \mathfrak{g}_{2} \stackrel{r}{\longrightarrow} \mathbb{R},$$ which is an invariant polynomial on $\mathfrak{g}_{1}$. Let us denote by $\overline{\chi}_{r} \in \Omega^{*}(P_{2})$ the Chern--Weil form constructed using the connection $\sigma_{2}$ and let $\overline{\chi}_{\widehat{\varphi}^{*}r} \in \Omega^{*}(P_{1})$ denote the Chern--Weil form constructed using the connection $\sigma_{1}$. It follows from the assumptions above that $\overline{\chi}_{\widehat{\varphi}^{*}r} = \Phi^{*}\overline{\chi}_{r}$ and consequently $$\chi_{\widehat{\varphi}^{*}r}(P_1) = \phi^{*} \chi_{r}(P_{2}) \in H^{*}_{\text{deR}}(B_{1}),$$ where $\phi\colon B_{1}\longrightarrow B_{2}$ denotes the base map of the bundle morphism $\Phi$. In particular, $\chi_{r}(P_{2}) \neq 0$ whenever $\chi_{\widehat{\varphi}^{*}r}(P_1) \neq 0$.\\

To sum up, if there exists a principal $G_{1}$--bundle $P_{1}$ with a bundle morphism $\Phi$ to $P_{2}$ such that the pullback under $\Phi$ of the characteristic class $\chi_{r}(P_{2})$ to $P_{1}$ is a non-trivial characteristic class of $P_{1}$, then $\chi_{r}(P_{2})$ is non-trivial and so the universal characteristic class $\chi_{r} \in H^{*}(BG_{2}; \mathbb{R})$ is also non-trivial.\\ 

Let us focus back on the Reznikov classes. Let $\cmfd$ be a closed connected coorientable contact manifold with a smooth action of a compact Lie group $G$ and assume that $\alpha$ is $G$--invariant (such an $\alpha$ can be obtained by averaging, see \cite[Proposition 2.8]{Ler1}). Denote by $\varphi\colon G\longrightarrow \SCont(M, \a)$ the corresponding Lie group homomorphism. Consider the extension by $\varphi$ of the universal principal $G$--bundle $EG\longrightarrow BG$, i.e.~the associated principal $\SCont(M, \a)$--bundle $$q\colon P = P\times_{\varphi} \SCont(M, \a)\longrightarrow BG$$ together with the bundle morphism $\Phi\colon EG\longrightarrow P$ defined by $\Phi(p) = [p, \text{id}]$. Now any principal connection 1--form $\sigma$ on $EG$ can be extended to a principal connection form $\theta\in\Omega^1(P,\mathrm{Lie}\SCont(M, \a))$ such that $\Phi^{*}\theta = \widehat{\varphi}\circ \sigma$. Applying the above procedure to the invariant polynomial $I_{k}$ we obtain a characteristic class $\chi_{\widehat{\varphi}^{*}I_{k}} \in H^{2k}(BG; \mathbb{R})$. By Cartan's Theorem \ref{thm:cartan} this class is uniquely determined by the invariant polynomial $\widehat{\varphi}^{*}I_{k}$ and hence it is non-trivial if and only if the polynomial is non-zero. Therefore in order to conclude the non-triviality of the Reznikov class $\chi_{k}$, it is enough to show that the invariant polynomial $\widehat{\varphi}^{*}I_{k}$ is a nonzero polynomial on the Lie algebra $\lieg$.\\

Let us describe the invariant polynomial $\widehat{\varphi}^{*}I_{k}$ in terms of the $\a$--moment map of the $G$--action $\varphi$ on $(M, \xi=\ker\a)$. The \emph{$\alpha$--moment map} is the map $m\colon M\longrightarrow \mathfrak{g}^{*}$ uniquely defined by $$\langle m(x), A \rangle = \alpha_{x}(\vec{A}(x)),$$ where $x \in M$, $A\in \mathfrak{g}$ and $\vec{A}\in \mathfrak{X}(M)$ is the fundamental vector field on $M$ induced by $A$. We also write the moment map as $m\colon M \times \mathfrak{g} \longrightarrow \mathbb{R}$.  Then for $A_{1}, \ldots, A_{k} \in \lieg$ we have
\begin{align*} (\widehat{\varphi}^{*}I_{k})(A_{1}, \ldots, A_{k}) = I_{k}\bigl(\widehat{\varphi}(A_{1}), \ldots, \widehat{\varphi}(A_{k}) \bigr) &=  I_{k}\bigl(\a(\vec{A}_{1}), \ldots, \a(\vec{A}_{k}) \bigr) = \\ &= \int_{M} m(-, A_{1})\cdot\ldots\cdot m(-, A_{k})\, \alpha\wedge(d\alpha)^{n}.
\end{align*}

In this situation we can prove that the even Reznikov classes are non-trivial, which is Proposition \ref{prop:evenRez}.\\

{\bf Proof of Proposition \ref{prop:evenRez}:}
Because the action $\varphi\colon G\longrightarrow \mathrm{Cont}(M, \alpha)$ is non-trivial by assumption, there exists $A\in \mathfrak{g}$ such that $\vec{A} \not\equiv 0$. In particular, $m(-, A) = \alpha(\vec{A}) \not\equiv 0$ since a contact vector field which lies completely in the contact distribution is necessarily the zero vector field. Consequently,
$$(\widehat{\varphi}^{*}I_{2l})(A, \ldots, A) = \int_{M} (m(-, A))^{2l}\,\alpha\wedge(d\alpha)^{n} >0$$
and $\widehat{\varphi}^{*}I_{2l}$ is a non-zero invariant polynomial on $\mathfrak{g}$. By Cartan's Theorem \ref{thm:cartan} the corresponding characteristic class is a non-trivial element of $H^{4l}(BG; \mathbb{R})$. Therefore so is the Reznikov class $\chi_{2l}$. \qed \\


For a special class of contact manifolds we can prove that also the odd Reznikov classes $\chi_{2l+1}$ are non-trivial. A contact manifold $(M, \xi)$ is called \emph{$K$--contact} if there exists a contact form $\alpha$ defining $\xi$ and a Riemannian metric $g$ such that the Reeb vector field $R$ is Killing with respect to $g$, i.e.~the Reeb flow of $\alpha$ is by isometries of $g$. The form $\alpha$ is then called $K$--contact too. An equivalent description of $K$--contact manifolds reads as follows.

\begin{proposition}[{\cite[Proposition 4.3]{Ler2}}]\label{prop:KCont} A closed connected contact manifold $(M, \xi)$ is $K$--contact if and only if $M$ admits an action of a torus $T$ by coorientation preserving contactomorphisms and there exists a vector $A\in \mathrm{Lie}\,T$ such that $m(-, A)\colon M\longrightarrow \mathbb{R}$ is a strictly positive function, where $m$ denotes the $\beta$--moment map of the $T$--action for any $T$--invariant contact form $\beta$.
\end{proposition}

Note that a $T$--invariant contact form always exists and then the $K$--contact form $\alpha$ is a multiple of $\beta$ by a $T$--invariant function, hence in particular $\alpha$ is also $T$-invariant. The proof of Proposition \ref{prop:evenRez} together with Proposition \ref{prop:KCont} imply Proposition \ref{prop:oddRez}.\\


Propositions \ref{prop:evenRez} and \ref{prop:oddRez} are rather general results on non-triviality of the Reznikov classes. However, if we are given a specific circle or torus action, more information can be recovered. This is illustrated in the following subsection.


\subsection{Explicit computations} \label{ssec:reznikov_calc} In the previous subsection we studied non-triviality of the Reznikov classes by pulling the back under under a $G$--action. In certain cases we can compute the pullbacks explicitly -- we need to understand the induced map $$ (B\varphi)^{*}\colon H^{*}(B\mathrm{SCont}(M, \a); \mathbb{R})\longrightarrow H^{*}(BG; \mathbb{R}) $$ since $\chi_{\widehat{\varphi}^*I_{k}} = (B\varphi)^{*}\chi_{k}$. The simplest case is that of contact manifolds admitting a free $\S^1$-action, which is the content of Theorem \ref{thm:reg_chlgy}.\\

{\bf Proof of Theorem \ref{thm:reg_chlgy}:} Fix $k\in \mathbb{N}$. We shall explicitly compute the invariant polynomial $\widehat{\varphi}^{*}I_{k}$. Let $R$ denote the Reeb vector field corresponding to the contact form $\a$. Then the homomorphism $\varphi$ is defined by $\varphi(s)(x) = \mathrm{Fl}^{R}_{s}(x)$ for $x \in M$ and the induced Lie algebra homomorphism $\widehat{\varphi}\colon \mathbb{R}\longrightarrow\mathrm{Lie}\,\SCont(M, \a) \cong C^{\infty}(M)^{R}$ reads as $$\widehat{\varphi}(t)(x) = \alpha_{x}\left(\left. \frac{\mathrm{d}}{\mathrm{d}s}\right|_{s=0} \mathrm{Fl}^{R}_{st}(x) \right) = \alpha_{x}(t\cdot R(x)) = t, $$ i.e.~$\widehat{\varphi}(t)$ is the constant function on $M$ with value $t$.\\

Therefore for $t\in \mathbb{R} = \mathrm{Lie}\,S^{1}$ we obtain $$(\widehat{\varphi}^{*}I_{k})(t, \ldots, t) = I_{k}(\widehat{\varphi}(t), \ldots, \widehat{\varphi}(t)) = \int_{M} t\cdot \ldots \cdot t\, \alpha\wedge(d\alpha)^{n} = t^{k}\cdot \mathrm{vol}_{\alpha}(M).$$  By Example \ref{ex:torus} this is a non-zero multiple of the invariant polynomial on $\mathrm{Lie}\,\S^{1}$ corresponding to the $k$-th power of the universal Euler class $e$. In other words, $(B\varphi)^{*}\chi_{k} = \mathrm{vol}_{\alpha}(M)\cdot e^{k}$, which proves the statement of the theorem. \qed \\

This precise description allows us to compute the order of the Reeb flow in the fundamental group of $\SCont(M,\alpha)$. Corollary \ref{cor:reg_htpy} states that the order is infinite.\\

{\bf Proof of Corollary \ref{cor:reg_htpy}:} Put $\mathcal{G} = \mathrm{SCont}_{0}(M, \alpha)$. Then $\pi_{0}(B\mathcal{G}) = \pi_{1}(B\mathcal{G}) = 0$ and $$\pi_{1}(\mathcal{G})\otimes \mathbb{R} \cong \pi_{2}(B\mathcal{G})\otimes \mathbb{R} \cong H_{2}(B\mathcal{G}; \mathbb{R})$$ by the Hurewicz theorem. Let $\varphi \colon \S^{1}\longrightarrow \mathcal{G}$ denote the homomorphism induced by the Reeb flow and consider the following commutative diagram, where all the horizontal arrows are isomorphisms.
\begin{center}
\begin{minipage}{0.8\linewidth}
\xymatrix{
H^{2}(B\mathcal{G}; \mathbb{R}) \ar[d]^{(B\varphi)^{*}} \ar[r] & \mathrm{Hom}(H_{2}(B\mathcal{G}; \mathbb{R}), \mathbb{R}) \ar[d]^{- \circ (B\varphi)_{*}} \ar[r] & \mathrm{Hom}(\pi_{1}(\mathcal{G})\otimes \mathbb{R}, \mathbb{R}) \ar[d]^{- \circ \varphi_{*}} \\
H^{2}(B\S^{1}; \mathbb{R}) \ar[r] & \mathrm{Hom}(H_{2}(B\S^{1}; \mathbb{R}), \mathbb{R}) \ar[r] & \mathrm{Hom}(\pi_{1}(\S^{1})\otimes \mathbb{R}, \mathbb{R}) \cong \mathbb{R}
}
\end{minipage}
\end{center} Let $c\in \pi_{1}(\S^{1}) \cong \mathbb{Z}$ be the positive generator and let $\gamma \in \mathrm{Hom}(\pi_{1}(\S^{1})\otimes \mathbb{R}, \mathbb{R})$ be the dual of $c$. By the proof of Theorem \ref{thm:reg_chlgy} the image of $(B\varphi)^{*}\chi_{1}$ under the bottom row of isomorphisms is a non-zero multiple of $\gamma$.\\

Now for any $k\in \mathbb{N}$ we have $\varphi_{*}(k\cdot c) = k\cdot [\text{Reeb}] \in \pi_{1}(\mathcal{G})$. But then $$\chi_{1}(k\cdot [\text{Reeb}]) = \chi_{1}(\varphi_{*}(k\cdot c )) = ((B\varphi)^{*}\chi_{1})(k\cdot c) = \text{const}\cdot \gamma(k\cdot c) = \text{const}\cdot k \neq 0.$$
In particular, $k \cdot [\text{Reeb}] \neq 0 \in \pi_{1}(\mathcal{G})$. \qed

\begin{remark} Corollary \ref{cor:reg_htpy} can be alternatively deduced from the homotopy invariance of  the Calabi--Weinstein functional on $\widetilde{\mathrm{Cont}}_{0}(M, \alpha)$. See \cite{Sh} and the references therein.
\end{remark}

Our next example is the case of a $\T^2$--action. Fix $n\in \mathbb{N}$ and consider the 3--torus $\mathbb{T}^{3} = \mathbb{R}^{3} / (2\pi\mathbb{Z})^{3}$ together with the contact form $\alpha_{n} = \cos(nt)dx - \sin(nt)dy$. This form is invariant under the $\T^{2}$--action defined by $(a,b)\cdot (x,y,t) = (x+a, y+b, t)$ and so we obtain a homomorphism $\varphi\colon \mathbb{T}^{2}\longrightarrow \mathrm{SCont}(\mathbb{T}^{3},\alpha_{n})$.

\begin{proposition}\label{prop:toric}
For $l\in \mathbb{N}$ the image of $\chi_{2l}$ under the induced map $$(B\varphi)^{*}\colon H^{*}(B\mathrm{SCont}(\mathbb{T}^{3}, \alpha_{n}); \mathbb{R}) \to H^{*}(B\mathbb{T}^{2}; \mathbb{R}) \cong \mathbb{R}[A, B]$$ is a nonzero multiple of $(A^{2}+B^{2})^{l}$ while the image of $\chi_{2l-1}$ is zero.
\end{proposition}

\begin{proof}
Consider the Lie algebra homomorphism $\widehat{\varphi}\colon \mathrm{Lie}\,\mathbb{T}^{2} \cong \mathbb{R}^{2} \to \mathrm{Lie}\,\mathrm{Cont}(\mathbb{T}^{3}, \alpha_{n}) \cong C^{\infty}(\mathbb{T}^{3})^{R}$ induced by $\varphi$. We shall compute the invariant polynomials $\widehat{\varphi}^{*}I_{k}$. First, for $(A, B) \in \mathbb{R}^{2}$ we have $$T_{(0,0)}\varphi(A, B)(x, y, t) = \left. \frac{\mathrm{d}}{\mathrm{d}\mathnormal{s}}\right|_{\mathnormal{s}=0} (x+s\cdot A, y+s\cdot B, t) = A\cdot \frac{\partial}{\partial x} +B\cdot \frac{\partial}{\partial y}$$ and so $$\widehat{\varphi}(A, B)(x,y,t) = (\alpha_{n})_{(x,y,t)}\left(A\cdot \frac{\partial}{\partial x} +B\cdot \frac{\partial}{\partial y}\right) = A\cdot \cos(nt) - B\cdot \sin(nt).$$

\noindent The volume form on $\mathbb{T}^{3}$ is given by $\alpha_{n} \wedge d\alpha_{n} = n\cdot dx\wedge dy\wedge dt$. Therefore, \begin{align*} (\widehat{\varphi}^{*}I_{k})((A, B), \ldots, (A,B)) = n \int_{\mathbb{T}^{3}} \bigl(A\cos(nt) &- B\sin(nt)\bigr)^{k}\,dx\wedge dy\wedge dt = \\ &= 4n\pi^{2} \int_{0}^{2\pi} \bigl(A\cos(nt) - B\sin(nt)\bigr)^{k}\,\mathrm{d}t.\end{align*}

Let $\tau \in [0, 2\pi)$ be such that $(\cos\tau, \sin\tau) = \frac{1}{\sqrt{A^{2}+B^{2}}}(A, B)$. Then we can rewrite the integral on the right-hand side as follows \begin{align*} \int_{0}^{2\pi} \bigl(A\cos(nt) - B\sin(nt)\bigr)^{k}\,\mathrm{d}t = (\sqrt{A^2+B^2})^{k} \int_{0}^{2\pi} \bigl(\cos(\tau)\cos(nt) - \sin(\tau)\sin(nt)\bigr)^{k}\,\mathrm{d}t = \\ = \left(\sqrt{A^2+B^2}\right)^{k} \int_{0}^{2\pi} \cos^{k}(nt+\tau)\,\mathrm{d}t  = \left(\sqrt{A^2+B^2}\right)^{k} \int_{0}^{2\pi} \cos^{k}(nt)\,\mathrm{d}t. \end{align*}

For $k=2l-1$ odd the last integral vanishes. But for $k=2l$ even
$c_{l} := \int_{0}^{2\pi} \cos^{2l}(nt)\,\mathrm{d}t >0$
because the integrand is an almost everywhere positive function. To sum up, $$(\widehat{\varphi}^{*}I_{k})((A, B), \ldots, (A, B)) = \begin{cases} 0 & \text{ if } k=2l-1, \\ 4n\pi^{2} c_{l}\cdot (A^{2}+B^{2})^{l} & \text{ if } k=2l \end{cases}$$ and the statement of the proposition follows.
\end{proof}

These computations are analogous to the results of A.\,G.\,Reznikov \cite[Section 1.4]{Re} for Hamiltonian characteristic classes. In particular, his calculation for the Fubini--Study symplectic form on $\C\P^{n}$ can be easily modified to prove the following proposition. This can be also generalized to the lens spaces $\S^{2n+1}/ \mathbb{Z}_{k}$ equipped with the standard contact form.

\begin{proposition}\label{prop:sphsur}
Let $\varphi\colon \mathrm{U}(n+1)\longrightarrow\mathrm{SCont}(\S^{2n+1}, \alpha_{\mathrm{st}})$ be the homomorphism given by the standard action of the unitary group on $\S^{2n+1} \subseteq \mathbb{C}^{n+1}$. Then the induced map $$(B\varphi)^{*}\colon H^{*}(B\mathrm{SCont}(\S^{2n+1}, \alpha_{\mathrm{st}}); \mathbb{R}) \to H^{*}(B\mathrm{U}(n+1);\mathbb{R})$$ is surjective.
\end{proposition}

The reader might have noticed that apart from the $3$--torus all our examples were prequantizations of integral symplectic manifolds. It is therefore natural to ask how the contact Reznikov classes are related to the original Hamiltonian Reznikov classes. This is will be clarified in the next subsection.


\subsection{Relationship to the Hamiltonian characteristic classes} \label{ssec:reznikov_ham} In this last subsection we briefly discuss the relationship between the contact Reznikov classes defined in Subsection \ref{ssec:reznikov} and the Hamiltonian Reznikov classes introduced in \cite{Re} in the case of prequantization spaces. Let us note that the Hamiltonian Reznikov classes can be also defined topologically using the so-called \emph{coupling class}. Using this approach they were studied by Gal, K\k{e}dra, McDuff and Tralle see e.g.~\cite{KeMc, GKT}.\\

Let $p\colon (M, \alpha)\longrightarrow (B, \omega)$ be a prequantization of an integral symplectic manifold $(B,\o)$ as in Section \ref{sec:preliminaries} or Subsection \ref{ssec:5SPH}. In particular, $\a$ is a contact form on $M$ such that $d\a = p^{*}\o$ and the Reeb flow of $\a$ generates a free $\S^{1}$-action. By Proposition \ref{prop:Vizman} there exists a Lie group homomorphism $q\colon \SCont_{0}(M, \a)\longrightarrow\Ham(B, \o)$.\\

The Lie algebra of $\Ham(B, \omega)$ is the algebra of Hamiltonian vector fields on $B$ and it can be identified via $\omega$ with the algebra of smooth functions on $B$ normalized with respect to the volume form $\omega^{n}$:
$$\mathcal{K}:= \mathrm{Lie}\Ham(B, \omega) \cong \left\{ h \in C^{\infty}(B)\,\Bigl|\, \int_{B}h\,\omega^{n} = 0\right\}.$$

For each $k\in \mathbb{N}$ the map $J_{k}\colon \mathcal{K}\times \ldots \times \mathcal{K}\longrightarrow \R$, $J_{k}(h_{1}, \ldots, h_{k}) = \int_{B} h_{1}\cdot\ldots\cdot h_{k}\,\omega^{n}$, is an invariant polynomial on $\mathcal{K}$ and so by the Chern--Weil theory it defines a characteristic class $\mu_{k} \in H^{*}(B\mathrm{Ham}(B, \omega), \R)$. These classes $\{\mu_k\}$ are called the \emph{Hamiltonian Reznikov classes}. Note that $\mu_{1}$ is trivial due to the normalization condition on $\mathcal{K}$.\\

Recall that the Lie algebra $\mathcal{L} := \mathrm{Lie}\SCont_{0}(M, \a)$ is isomorphic to the algebra of smooth functions on $M$ which are invariant under the Reeb flow of $\alpha$. Because the Reeb flow of $\a$ is exactly the $\S^{1}$--fibre of $p$, every such function $H\in C^{\infty}(M)$ descends to a unique function $h\in C^{\infty}(B)$ determined by $H = h\circ p$. It can be easily verified that the Lie algebra homomorphism $\widehat{q}\colon \mathcal{L}\longrightarrow \mathcal{K}$ is then given by $$\widehat{q}(H) =  h - \frac{\int_{B}h\,\omega^{n}}{\int_{B}\omega^{n}},$$ i.e.~$h$ normalized with respect to the volume form $\o^{n}$.\\

We need to understand the invariant polynomials $\widehat{q}^{*}J_{k}$ on $\mathcal{L}$. First note that for $H\in \mathcal{L}$ integrating along the $\S^{1}$--fibre we obtain $$ \int_{M}H\,\a\wedge(d\a)^{n} = \int_{M}(h\circ p)\,\a\wedge (p^{*}\o)^{n} = 2\pi\cdot \int_{B}h\,\o^{n}.$$ Put $ c:= \left( 2\pi\cdot \int_{B}\o^{n}\right)^{-1}$. Then for any $k \geq 2$ and $H_{1}, \ldots, H_{k}\in \mathcal{L}$ we have \begin{align*} {\widehat{q}}^{*}&J_{k}(H_{1}, \ldots, H_{k}) =  \int_{B} \bigl(h_{1}-c\cdot I_{1}(H_{1})\bigr)\cdot \ldots \cdot \bigl(h_{k}-c\cdot I_{1}(H_{k})\bigr)\,\o^{n} = \\ &= \int_{B} h_{1}\cdot \ldots \cdot h_{k}\,\o^{n} + r_{k}(h_{1}, \ldots, h_{k}) = \frac{1}{2\pi}\cdot \int_{M} H_{1}\cdot \ldots \cdot H_{k}\,\a\wedge(d\a)^{n} + R_{k}(H_{1}, \ldots, H_{k}) = \\ &= \frac{1}{2\pi}\cdot I_{k}(H_{1}, \ldots, H_{k}) + R_{k}(H_{1}, \ldots, H_{k}). \end{align*} The rest $R_{k}$ is an invariant polynomial on $\mathcal{L}$ in $k$ variables and it is an algebraic combination of the 'lower order' invariant polynomials $I_{1}, \ldots, I_{k-1}$.\\

We are now ready to state the last result of the paper. 

\begin{proposition}
The induced map $(Bq)^{*}\colon H^{*}(B\mathrm{Ham}(B, \o );\mathbb{R})\longrightarrow H^{*}(B\mathrm{SCont}_{0}(M, \a );\mathbb{R})$ sends the subalgebra generated by the Hamiltonian Reznikov classes $\mu_{2}, \ldots, \mu_{k}$ injectively into the subalgebra generated by the contact Reznikov classes $\chi_{1}, \chi_{2}, \ldots, \chi_{k}$.
\end{proposition}
\begin{proof} The statement of the proposition follows from the discussion above and the following commutative diagram
\begin{center}
\begin{minipage}{0.8\linewidth}
\xymatrix{
I^{*}(\Ham(B, \o) ) \ar[d] \ar[r]^{(dq)^{*}} & I^{*}(\SCont_{0}(M, \a) ) \ar[d] \\
H^{*}(B\mathrm{Ham}(B, \o );\mathbb{R}) \ar[r]^{(Bq)^{*}} & H^{*}(B\mathrm{SCont}_{0}(M, \a );\mathbb{R})
}
\end{minipage}
\end{center}
where $I^{*}(-)$ denotes the ring of invariant polynomials and the vertical arrows are the Chern--Weil homomorphisms.
\end{proof}


\end{document}